\documentclass[11pt,reqno]{amsart}

\usepackage{amscd,amssymb,amsmath,amsthm}
\usepackage{graphicx}
\usepackage{color}
\usepackage{cite}
\usepackage[shortlabels]{enumitem}

\newtheorem{thm}{Theorem}

\newtheorem{lemma}{Lemma}
\newtheorem{pro}{Proposition}
\newtheorem{rk}{Remark}
\newtheorem{cor}{Corollary}
\newtheorem{con}{Conjecture}

\numberwithin{equation}{section} 

\def\l{\lambda}

\def\R{\mathbb{R}}

\def\Z{\mathbb{Z}}

\begin{document}
\title[Hard and Soft-Core Widom-Rowlinson models]{
Hard-Core and Soft-Core Widom-Rowlinson models on Cayley trees
}

\author{S. Kissel, C. K\"ulske, U. A. Rozikov}

\address{S. Kissel\\ Fakult\"at f\"ur Mathematik,
Ruhr-University of Bochum, Postfach 102148,\,
44721, Bochum,
Germany.}
\email{sascha.kissel@ruhr-uni-bochum.de}

\address{C.\ K\"ulske\\ Fakult\"at f\"ur Mathematik,
Ruhr-University of Bochum, Postfach 102148,\,
44721, Bochum,
Germany.}
\email {Christof.Kuelske@ruhr-uni-bochum.de}

\address{U.\ A.\ Rozikov\\ Institute of mathematics,
81, Mirzo Ulug'bek str., 100125, Tashkent, Uzbekistan.}
\email {rozikovu@yandex.ru}

\begin{abstract} We consider both Hard-Core and Soft-Core Widom-Rowlinson models
with spin values  $-1,0,1$ on a Cayley tree of order $k\geq 2$ and we are interested
in the Gibbs measures of the models.
The models depend on 3 parameters: the order $k$ of the tree, $\theta$
describing the strength of the (ferromagnetic or antiferromagnetic) interaction, and
$\lambda$ describing the intensity for particles. The Hard-Core
Widom-Rowlinson model corresponds to the case $\theta=0$. \\
For the binary tree $k=2$, and for $k=3$ we prove that the ferromagnetic model
has either one or three splitting Gibbs measures
(tree-automorphism invariant Gibbs measures (TISGM)
which are tree-indexed Markov chains). We also give the exact form of the
corresponding  critical curves $\lambda_{\rm cr}(k,\theta)$ in parameter space.
For higher values of $k$ we give an explicit
sufficient bound ensuring non-uniqueness which we conjecture to be the
exact curve.
Moreover, for the antiferromagnetic model we explicitly give two 
critical curves $\lambda_{{\rm cr},i}(k,\theta)$, $i=1,2$, and prove that on these curves there are exactly 
two TISGMs; between these curves there are exactly three TISGMs; otherwise there exists a unique TISGM. 
Also some periodic and non-periodic SGMs are constructed in the ferromagnetic model. 
\end{abstract}
\maketitle



\section{Introduction}
The Widom-Rowlinson model has been introduced 
	by \cite{Wi}
	as a model for point particles which carry charges plus or minus one, 
	with positions in Euclidean space. In the original Hard-Core version
	the interaction strictly forbids particles of opposite signs to becomes closer than a fixed 
	radius. The continuum Widom-Rowlinson model shows a provable phase 
	transition at high enough equal intensity for plus and minus particles. The equilibrium properties 
	have been investigated in \cite{CCK}, \cite{Ru}, \cite{Wi}. For 
	the behavior  under stochastic spin-flip dynamics with a view to Gibbs-non Gibbs transitions, 
	see  \cite{JK}. 
	Related versions of such Hard-Core models, 
	have be studied on lattices (see \cite{GL},
	\cite{HT}, \cite{MSSZ}, \cite{MSS}). 
	The Hard-Core constraint for a Widom-Rowlinson model 
	on a lattice, or a graph, like a tree, means that particles of 
	opposite signs are forbidden to appear next to each other on neighboring sites of the graph. 
	
	Studies of multicolor hardcore models with rich classes of interactions on trees can be found in \cite{Ro}.
	When the Hard-Core constraint is relaxed, we come to Soft-Core models, which are more difficult 
	to analyze as they are depending on another parameter,  
	which governs the strength of the repulsion between particles of opposite signs (in the ferromagnetic case) 
	or the attraction (in the antiferromagnetic case).

Let us more specifically to the model on trees. 
The Hard-Core Widom-Rowlinson (HCWR) model considered in a part of 
this paper is identical to the
hinge constraint model of \cite{Br1} (see also \cite{Br4}, \cite{XR1}, \cite{MRS},
\cite{RSh}-\cite{RKh}, \cite{Wh}). In these papers the tree automorphism invariant splitting Gibbs measures (TISGMs) 
are investigated on the Cayley tree
of order $k\geq 2$, transition temperatures are computed, and also some periodic and non-periodic splitting
Gibbs measures are constructed. 

%

The methods of these papers were based on the description of boundary laws  which are in one-to-one
correspondence with the splitting Gibbs measures. 
The boundary laws of WR-model are two dimensional vectors with positive entries which satisfy a
non-linear fixed-point equation (tree recursion).  A given boundary law defines
the transition matrix of the corresponding tree-indexed 
Markov chain (see \cite[Chapter 12]{Ge} and \cite{Ro} for detailed definitions). All extremal Gibbs measures are splitting Gibbs measures (see \cite[Theorem 12.6]{Ge}), 
therefore,  if  there is only one splitting  Gibbs
measure, then there is only one Gibbs measure of any kind.
To decide the converse, namely whether a given Gibbs measure 
is extremal, is a separate difficult problem (see \cite{FK},\cite{KR},\cite{LN},\cite{Sl}).

In this paper we focus on the study of tree-automorphism invariant 
splitting Gibbs measures for the Soft-Core version of the Widom-Rowlinson model (SCWR). 
We review also some of the Hard-Core results which are rediscovered as special cases.

The paper is organized as follows. In Section 2 we give the main definitions and formulation of the problem. Section 3 contains
a compatibility condition, i.e an equation for boundary laws. Section 4 is devoted to TISGMs and we divide this section
to several subsections under some conditions on the parameters of the model. For any $k\geq 2$ we give explicit regions of
parameters of non-uniqueness of TISGMs.  In Section 5, for $k=2$, and $3$ we give upper and lower bounds of the
boundary laws. The maximal and minimal boundary laws then define extreme TISGMs. Sections 6 and 7 are devoted to some periodic and non-periodic splitting GMs.

\section{Definitions and formulation of the problem}
		
	Let $\mathbb{T}^k$, $k\geq 1$ be a rooted Cayley tree. Let $d(i,j)$, $i,j\in \mathbb T^k$7
the distance between vertices $i,j$, i.e. the number of edges of the shortest path connecting $i$ and $j$.

By $0$ we denote the root of the tree and define
$$W_n=\{j \in \mathbb{T}^k\,:\, d(0,j)=n\}, \ \
V_n = \cup_{j=0}^n  W_j.$$
The set $S(i)$ is the set of direct successors
of a vertex $i\in \mathbb T^k$, i.e., for $i\in W_n$ we have
$$S(i)=\{j\in W_{n+1}: d(i,j)=1\}.$$
See \cite[Chapter 1]{Ro} for algebraic properties of the Cayley tree.

 The configuration space is given by $\Omega := \{-1,0,1\}^{\mathbb{T}^k}$.
 We denote elements of $\Omega$ by $\omega$, $\sigma$, etc.
Thus a configuration is a function $\omega: i\in \mathbb T^k\to \omega_i\in \{-1,0,1\}$.

Denote by $\Omega_{A}$ the set of all configurations
on the set $A\subset \mathbb T^k$.

The Hamiltonian for the SCWR-model of step $n$, i.e. on the configuration set $\Omega_{V_n}$, is given by
	\begin{align*}
	H^{sc}_n(\omega) = J \sum_{\{i,j\}\in L_n} \mathbf{1}(\omega_i\omega_j =-1) - {\ln(\lambda)\over \beta} \sum_{i\in V_n} \omega_i^2 .
	\end{align*}
where $\beta={1\over T}$ and $T>0$ is the temperature.	
The parameter $J\in \mathbb R$ can be seen as repulsion or attraction between particles of different charges depending on the sign of $J$ and $\lambda>0$ as an activity.

	The associated finite volume Gibbs measure on $\Omega_{V_n}$ with external
fields on the boundary, $$\{h^i=(h_{-1,i},h_{0,i},h_{1,i}) \in \mathbb{R}^3 \,:\, i\in \mathbb{T}^k\} $$ is defined by
	\begin{align*}
	\mu_{n,\beta}(\omega) = \frac{1}{Z_n}\exp\left(-\beta H^{sc}_n(\omega) +\sum_{i\in W_n} h_{\omega_i,i}\right).
	\end{align*}
The sequence $\mu_{n,\beta}$, $n\geq 1$ is said to be compatible
if for all $n\geq2$ and all $\sigma^{n-1}\in \Omega_{V_{n-1}}$
	\begin{align*}
	\sum_{\omega^n\in \Omega_{W_n}} \mu_{n,\beta}(\sigma^{n-1}\omega^n)  = \mu_{n-1,\beta}(\sigma^{n-1})
	\end{align*}
	holds. Here

$$(\sigma^{n-1}\omega^n)_i=\left\{\begin{array}{ll}
\sigma^{n-1}_i, \ \ \mbox{if} \ \ i\in V_{n-1}\\[2mm]
\omega^n_i, \ \ \mbox{if} \ \ i\in W_n
\end{array}.\right.
$$

For a sequence of compatible finite volume Gibbs measures $(\mu_{n,\beta})_n$, by Kolmogorov's extension theorem,
 there exists a unique measure $\mu$ defined on the whole tree with
 $\mu(\sigma\big|_{V_n}=\sigma^n)= \mu_{n,\beta}(\sigma^n)$ .
Following \cite{Ro12}, we call $\mu$ a splitting Gibbs measure.

 The finite volume Gibbs measure for the Hard-Core WR-model
 we get for $J>0$ by $\lim_{\beta\rightarrow\infty} \mu^{sc}_{n,\beta} = \mu^{hc}_{n}$.

 \section{Compatibility equations}
 The following theorem gives conditions to make the finite volume Gibbs measures
 compatible.  The proof is included for convenience of the reader, the compatibility relations 
 can also be obtained by 
 an application of  Theorem 12.12 of \cite{Ge}, together with Definition 12.10, to our model. 
 
	\begin{thm}\label{cc}
		Let $0<\beta<\infty$, $J\in \mathbb R$ and $\lambda>0$.
		The sequence of probability measures $(\mu^{sc}_{n,\beta})_n$ is compatible if and only if for any $x\in \mathbb T^k$ the following two equations hold
		\begin{equation}\label{eq: comp}
\begin{array}{ll}
		\tilde h_{+,i} = \ln(\lambda)+\sum_{j\in S(i)} f(\tilde h_{+,j},\tilde h_{-,j},\theta)\\[3mm]
\tilde h_{-,i} = \ln(\lambda)+\sum_{j\in S(i)} f(\tilde h_{-,j},\tilde h_{+,j},\theta),
\end{array}
		\end{equation}
 where $\theta=\exp(-J\beta)$,  $f(x,y,\theta)= \ln(\frac{1+e^{x}+\theta e^{y}}{1+e^{x}+e^{y}})$
and $\tilde h_{\pm,j}= \ln(\lambda)+h_{\pm1,j}-h_{0,j}$.
	\end{thm}
	\begin{proof}
		First we show that compatibility implies \eqref{eq: comp}. By this we get
		\begin{align*}
		\frac{Z_{n-1}}{Z_n} \sum_{\omega_n\in\Omega_{W_n}} \exp\left(\sum_{i\in W_{n-1}}\sum_{j\in S(i)} \left(-J\beta \mathbf{1}(\sigma_i\omega_j =-1) +\ln(\lambda) \omega_j^2 +  h_{\omega_j,j}\right)\right) = \exp(\sum_{i\in W_{n-1}} h_{\sigma_i,i}).
		\end{align*}
Fix arbitrary $i\in W_{n-1}$ and consider arbitrary configurations $\sigma_{W_{n-1}}$, with fixed $\sigma_{i}$ to be one of $1,-1$ and $0$.
Then we get the following three equations
		\begin{align*}
		\frac{Z_{n-1}}{Z_n} \sum_{\omega_n\in\Omega_{W_n}} \exp\left(\sum_{i\in W_{n-1}}
\sum_{j\in S(i)} \left(-J\beta \mathbf{1}(\omega_j =\mp 1) +\ln(\lambda) \omega_j^2 +
h_{\omega_j,j}\right)\right) &= \prod_{i\in W_{n-1}}\exp(h_{\pm 1,i})\\
		\frac{Z_{n-1}}{Z_n} \sum_{\omega_n\in\Omega_{W_n}}
\exp\left(\sum_{i\in W_{n-1}}\sum_{j\in S(i)} \left( \ln(\lambda) \omega_j^2 +  h_{\omega_j,j}\right)\right) &= \prod_{i\in W_{n-1}}\exp(h_{0,i}).
		\end{align*}
		Dividing the first two equations by the last one and using some combinatorial arguments yields
		\begin{align*}
	\prod_{j\in S(i)}	\left(\frac{\sum_{q\in \{-1,0,1\} } \exp( -J\beta \mathbf{1}(q=\mp 1) +\ln(\lambda) q^2 +  h_{q,j})}{\sum_{q\in \{-1,0,1\} } \exp(  \ln(\lambda) q^2 +  h_{q,j})}\right) = e^{h_{\pm1,i}-h_{0,i}}.
		\end{align*}
		Using the substitution defined above we obtain
		\begin{align*}
		e^{\tilde h_{\pm,i}-\ln(\lambda)} = \prod_{j\in S(i)} \left(\frac{1+e^{\tilde h_\pm}+e^{-J\beta+\tilde h_\mp}}{1+e^{\tilde h_\pm}+e^{\tilde h_\mp}}\right).
		\end{align*}
		By using the logarithm and adding $\ln(\lambda)$ on both sides \eqref{eq: comp} follows.\\
		For the second implication we split the Hamiltonian into a part which depends only on the configuration up to the step $n-1$ and one depending on step $n-1$ and $n$. By definition this yields
		\begin{align*}
		\sum_{\omega_n\in \Omega_{W_n}}\mu_n (\sigma_{n-1}\omega_n) =\frac{1}{Z_n} e^{-\beta H_{n-1}^{sc}(\sigma_{n-1})} \prod_{i\in W_{n-1} } \prod_{j\in S(i) } \left(\sum_{q\in \{-1,0,1\}} e^{-\beta\mathbf{1}(q\sigma_i=-1)+\ln(\lambda)q^2+h_{q,j} }\right).
		\end{align*}
		Using again the substitution for $\tilde h$ one get from \eqref{eq: comp} the equations
		\begin{align*}
		a(i) e^{h_{\pm 1,i}} =  \prod_{j\in S(i)}(\sum_{q\in \{-1,0,1\} } \exp( -J\beta \mathbf{1}(q=\mp 1) +\ln(\lambda) q^2 +  h_{q,j})).
		\end{align*}
		for some function $a$ which is bigger than zero. These equations together yields
		\begin{align*}
		\sum_{\omega_n\in \Omega_{W_n}}\mu_n (\sigma_{n-1}\omega_n)  = \frac{\prod_{i\in W_{n-1}} a(i)}{Z_n} e^{-\beta H_{n-1}^{sc}(\sigma_{n-1})} \prod_{i\in W_{n-1} }  e^{h_{\sigma_i,i}}
		\end{align*}
		and since $\mu_{n,\beta}^{sc}$ is a probability measure it follows that
		\begin{align*}
		1 = \sum_{\sigma_{n-1}\in V_{n-1}}\frac{\prod_{i\in W_{n-1}} a(i)}{Z_n} e^{-\beta H_{n-1}^{sc}(\sigma_{n-1})} \prod_{i\in W_{n-1} }  e^{h_{\sigma_i,i}}.
		\end{align*}
		This implies $\prod_{i\in W_{n-1}} a(i)= \frac{Z_n}{Z_{n-1}}$ and concludes the proof.
	\end{proof}
	\begin{cor}
			Let $J>0$ and $\lambda>0$.
			The sequence of probability measures $(\mu^{hc}_{n})_n$ is compatible if and only if for any $x\in \mathbb T^k$ the two equations
			\begin{align}
			\tilde h_{+,i} = \ln(\lambda)+\sum_{j\in S(i)} g(\tilde h_{+,j},\tilde h_{-,j})\\
\tilde h_{-,i} = \ln(\lambda)+\sum_{j\in S(i)} g(\tilde h_{-,j},\tilde h_{+,j})
			\end{align}
			holds, where $g(x,y)= \ln(\frac{1+e^{x}}{1+e^{x}+e^{y}})$ and $\tilde h_{\pm,j}= \ln(\lambda)+h_{\pm1,j}-h_{0,j}$.
	\end{cor}
	\begin{proof}
		Follows from Theorem \ref{cc} with $\theta=0$.
	\end{proof}

\section{Translational invariant Gibbs measures}

In particular, we are interested in the translation-invariant spitting Gibbs measures (TISGMs).
In this case the external field vectors $h^i$ do not depend on $i$, i.e. $h^i = h$ for all $i\in \mathbb{T}^k$ and some $h\in \R^3$.
So the equations \eqref{eq: comp} by introducing two new variables $x= e^{\tilde h_+}$ and $y= e^{\tilde h_-}$ can be written as
 \begin{equation}\label{exy}
 \begin{array}{ll}
 x = \lambda\left(\frac{1+x+\theta y}{1+x+y}\right)^k\\[2mm]
 y = \lambda\left(\frac{1+\theta x+y}{1+x+y}\right)^k.
  \end{array}
  \end{equation}
  In the case of HCWR-model (i.e. $\theta=0$) the following theorem is known

  \begin{thm} \cite{RKh} Let $k\geq 2$ and
$\lambda_{cr}(k)={1\over k-1}\cdot\left({k+1\over k}\right)^k.$ Then
\begin{itemize}
\item[1.] For $\lambda>\lambda_{cr}(k)$ there exist at least three
TISGMs,

\item[2.] For $\lambda\leq\lambda_{cr}(k)$ there exists a unique
TISGM.
\end{itemize}
\end{thm}
\begin{rk} For the Hard-Core model on the Cayley tree of order two (i.e. $k=2$)
    it is proven that for $\lambda\leq \frac{9}{4}$ there exists only
    one TISGM and for $\lambda>\frac{9}{4}$ there are exactly $3$ (see \cite{RSh}, \cite{Ro}).
    Such a result is also true for $k=3$: if $\lambda>\lambda_{cr}={32\over 27}$ (see \cite{XR1}) then there exist {\rm exactly}
three TISGMs.
\end{rk}
  The following lemma is obvious:
  \begin{lemma}\label{le} If $(x^*,y^*)$ is a solution to (\ref{exy}) then $(y^*,x^*)$ is also its solution.
  \end{lemma}
  In particular, from this lemma it follows that if there exists a solution, $(x^*,y^*)$, with $x^*\ne y^*$,
  then the equation has more than one solutions.

  Subtracting from the first equation the second one we get
  $$(x-y)\left[1-\lambda(1-\theta){(1+x+\theta y)^{k-1}+\dots+(1+\theta x+y)^{k-1}\over (1+x+y)^k}\right]=0.$$
   From this we get $x=y$ or
  \begin{equation}\label{2h}
  (1+x+y)^k=\lambda(1-\theta)((1+x+\theta y)^{k-1}+\dots+(1+\theta x+y)^{k-1}).
  \end{equation}

  {\it CASE: $J<0$}. In this case we have $\theta=\exp(-J\beta)>1$. Therefore (\ref{2h}) is not satisfied, since the LHS is positive and
  RHS is negative. Thus we have only $x=y$. Therefore, in this case Lemma \ref{le} can not be applied to show non-uniqueness.

  Then from the first equation of (\ref{exy}) we get
  \begin{equation}\label{ef}
  \lambda^{-1}x=\left(\frac{1+(1+\theta)x}{1+2x}\right)^k.
  \end{equation}
  Denoting
  $$a= {2^k\over \lambda(1+\theta)^{k+1}}, \ \ b={1+\theta\over 2}, \ \ t=(1+\theta)x.$$

  Then equation (\ref{ef}) can be rewritten as
\begin{equation}\label{bu1}
at=\left({1+t\over b+t}\right)^k.
\end{equation}
The detailed analysis of solutions to equation (\ref{bu1}) is given in \cite{Pr}, Proposition 10.7,
which is the following:

\begin{pro}\label{Prp} Equation (\ref{bu1}) with $t\geq 0$,
$k\geq 1$, $a,b >0$ has
	a unique solution if either $k=1$ or
	$b\leq ({k+1 \over k-1})^2$.
 If $k>1$ and $b>({k+1 \over k-1})^2$
then there exist $\nu_1(b,k)$, $\nu_2(b,k)$, with
$0<\nu_1(b,k)< \nu_2(b,k)$, such that \begin{itemize}
	\item[1.] the equation has three
	solutions if $\nu_1(b,k)<a< \nu_2(b,k)$,
	\item[2.] the equation has two solutions if either
	$a=\nu_1(b,k)$ or $a= \nu_2(b,k)$,
	\item[3.] the equation has one solution if $a\notin  [\nu_1(b,k),\nu_2(b,k)]$.
\end{itemize}  In fact:
$$ \nu_i(b,k)={1\over x_i}\left({1+x_i \over b+x_i}\right)^k, $$
where $x_1,x_2$ are the solutions of
$$ x^2+[2-(b-1)(k-1)]x+b=0.$$
\end{pro}

By Proposition \ref{Prp} the equation (\ref{ef}) has a unique solution if $k=1$ or $k\ne 1$ and
\begin{equation}\label{kk}
{1+\theta\over 2}\leq \left(k+1\over k-1\right)^2.
\end{equation}
Thus the critical value of $\theta$ for non-uniqueness is found from the equation ${1+\theta\over 2}=\left(k+1\over k-1\right)^2$:
$$ \theta_{\rm cr}=2\left(k+1\over k-1\right)^2-1>1.$$
Using Proposition \ref{Prp}, for given $k\geq 2$ and $\theta>\theta_{\rm cr}$
 we define two critical values for $\lambda$:
$$\lambda_{{\rm cr},i}:=\lambda_{{\rm cr},i}(k,\theta)={2^kx_i\over (1+\theta)^{k+1}}\left({1+\theta+2x_i\over 2(1+x_i)}\right)^k, \, i=1,2.$$
Here $x_1$ and $x_2$ are solutions of the following quadratic equation:
$$ 2x^2+[4-(\theta-1)(k-1)]x+\theta+1=0.$$

Summarizing we obtain
\begin{thm}\label{t>} For the SCWR-model in the antiferromagnetic case $J<0$ the following assertions hold:
\begin{itemize}
\item[1)] If $\theta\leq\theta_{\rm cr}$ then there exists a unique translation invariant splitting Gibbs measure (TISGM).
\item[2)] If $\theta>\theta_{\rm cr}$ then
\begin{itemize}
\item[2.a)] if $\lambda\in (0,\lambda_{{\rm cr},2})\cup (\lambda_{{\rm cr},1}, +\infty)$ then there is unique TISGM.
\item[2.b)] if $\lambda\in \{\lambda_{{\rm cr},2}, \lambda_{{\rm cr},1}\}$ then there are two TISGMs.
\item[2.c)] if $\lambda\in (\lambda_{{\rm cr},2}, \lambda_{{\rm cr},1})$ then there are three TISGMs.
\end{itemize}
\end{itemize}
\begin{figure}[!htb]
\includegraphics[width=0.7\textwidth]{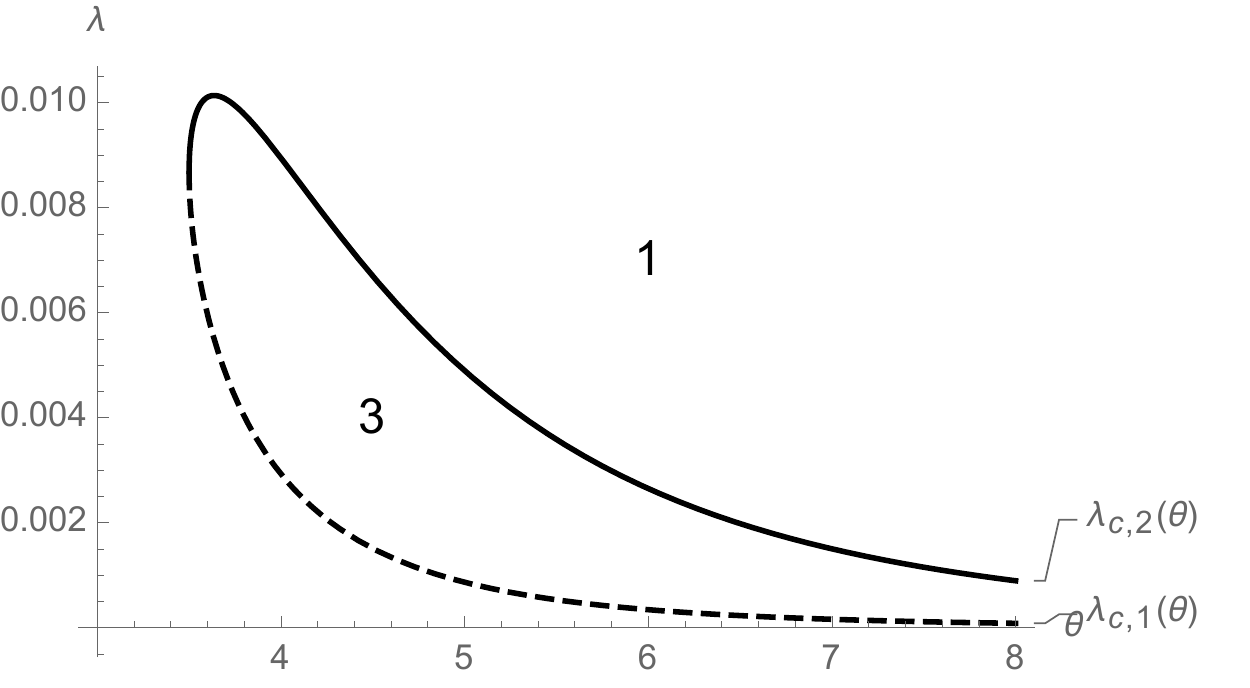}
\caption{Curves of critical $\lambda$'s for $k=5$. Numbers of TISGMs are shown on the related domains. On the curve there are exactly two TISGMs.}
\end{figure}
\end{thm}
  {\it CASE: $J>0$.} In this case the condition (\ref{kk}) is always satisfied, i.e. we have the following
  \begin{pro}\label{p<} For the ferromagnetic SCWR-model, for any $k\geq 2$, $\lambda>0$ the
  system (\ref{exy}) has a unique solution on the line $y=x$. Denote this solution by $(x^*,x^*)$, where $x^*$ is a function of all parameters.
\end{pro}

 {\it SubCASE: $J>0$, $k=2$.}  We follow the approach of \cite{RSh}
  to find the TISGMs for the Soft-Core case.
	Assume $k=2$. Then from (\ref{2h}) (i.e $x\neq y$) we get
	\begin{align*}
		(1+x+y)^2= \lambda (1-\theta)[1-\theta+(1+\theta)(1+x+y)].
	\end{align*}
	Solving this equation as quadratic polynomial in $x$ we get
	\begin{equation}\label{ee}
1+x+y=\frac{ \lambda(1-\theta^2)+(1-\theta) \sqrt{\lambda ^2 (1+\theta)^2+4 \lambda}}{2}:=g(\theta,\lambda).
	\end{equation}
	Putting this formula into first equation of (\ref{exy}) we obtain
\begin{equation}\label{sqe}
		x = \lambda \Bigl(\frac{1+x+\theta(g(\theta,\lambda)-1-x)}{g(\theta,\lambda)}\Bigr)^2.
	\end{equation}
Define
$$\theta_c(2):={1\over 3}, \ \ \ \lambda_{\rm cr}(2):= {1\over 1-3\theta}\cdot \left({3\over 2}\right)^2.$$
Simple but long calculations show that the equation (\ref{sqe}):
\begin{itemize}
\item has no positive solution if $\theta\geq \theta_c(2)$.
\item if $\theta<\theta_c(2)$ then
\begin{itemize}
\item[a.]  for $\lambda<\lambda_{\rm cr}(2)$ the equation has no positive solution.
\item[b.]  for $\lambda=\lambda_{\rm cr}(2)$ the equation has a unique positive solution. Denote it by $x^*$.
\item[c.]  for $\lambda>\lambda_{\rm cr}(2)$ there exist two positive solutions. Denote them by $x_1^*$ and $x^*_2$, with $x_1^*<x^*_2$.
\end{itemize}
\end{itemize}
By (\ref{ee}) we can find $y^*$, $y_1^*$ and $y^*_2$
corresponding to $x^*$, $x_1^*$ and $x^*_2$ respectively.
In fact, using Vieta's formulas applied to quadratic polynomial, or just using symmetry of $x,y$ one can
see that
$$x^*=y^*, \ \ x^*_1=y^*_2, \ \ x^*_2=y^*_1.$$

Consequently, we have up to three solutions of (\ref{exy}):
\begin{equation}\label{ye}
\mathcal M=\{(x^*,x^*), \ \ (x_1^*,x_2^*), \ \ (x^*_2, x^*_1)\}.
\end{equation}

Summarize now results of this subsection in the following
\begin{thm}\label{t<} For the ferromagnetic SCWR-model on the binary tree $k=2$ the following assertions hold:
\begin{itemize}
\item[1)] If $\theta\geq\theta_{c}(2)$ then there exists a unique TISGM.
\item[2)] If $\theta<\theta_{c}(2)$ then
\begin{itemize}
\item[2.a)] if $\lambda\leq\lambda_{{\rm cr}}(2)$ then there is a unique TISGM, denoted by $\mu^*$.
\item[2.b)] if $\lambda>\lambda_{{\rm cr}}(2)$ then there are three TISGMs, denoted by $\mu^*$, $\mu^*_i$, $i=1,2$.
\end{itemize}
\end{itemize}
\end{thm}
Note that in each case of Theorem \ref{t<} one of the TISGMs corresponds to the unique solution mentioned in Proposition \ref{p<}. 	
Note for $\beta=\infty$, i.e. $\theta=0$ this fits to the results known for the Hard-Core model (see \cite{RSh}, \cite{Ro}).\\

{\it SubCASE: $J>0$, $k\geq 3$.} In this case we shall find explicit values
of $\theta_{\rm c}=\theta_{\rm c}(k)$ and $\lambda_{\rm c}=\lambda_{\rm c}(k,\theta)$ such that
if $\theta<\theta_{\rm c}$ and $\lambda>\lambda_{\rm c}$ then there are at least three TISGMs.

Denoting $\sqrt[k]{x}=u>0, \ \sqrt[k]{y}=v>0$, $\sqrt[k]{\lambda}=a$  we obtain the
system of equations

\begin{equation}\label{rus3.5} \left\{\begin{array}{ll}
u=a{1+ u^k+\theta v^k \over 1+u^k+v^k},\\[3 mm]
v=a{1+ \theta u^k+v^k\over 1+u^k+v^k}
\end{array}\right.
\end{equation}
from the system (\ref{exy}).

Since $\theta<1$,  if $(u,v)$ is a solution to (\ref{rus3.5}) then we have
\begin{equation}\label{uav}
a\theta< u<a, \ \ \ a\theta<v<a.
\end{equation}

{\it SubsubCASE: $J>0$, $k=3$.}  Denote sum and product of $u$ and $v$ by
\begin{equation}\label{sp}
s=u+v, \ \ p=uv.
\end{equation}
Dividing the first equation of (\ref{rus3.5}) by the second one (for $k=3$) we get (since $u\ne v$)
\begin{equation}\label{3s}
{u\over v}={1+u^3+\theta v^3\over 1+\theta u^3+v^3} \ \ \Rightarrow \ \ 1+\theta(u^3+v^3)+(\theta-1)(u^2v+uv^2)=0.
\end{equation}
Using the representations
$$u^3+v^3=s(s^2-3p), \ \  u^2v+uv^2=sp,$$
we get from (\ref{3s})  that
\begin{equation}\label{p}
p={1+\theta s^3\over (1+2\theta)s}.
\end{equation}
Adding the first equation of (\ref{rus3.5}) to the second one (for $k=3$) we get
$$u+v=a\cdot {2+(1+\theta)(u^3+v^3)\over 1+u^3+v^3}.$$
This equality by using (\ref{sp}) can be represented as

\begin{equation}\label{s}
s=a\cdot {2+(1+\theta)s(s^2-3p)\over 1+s(s^2-3p)}.
\end{equation}
Therefore the system (\ref{rus3.5}) is represented as system of equations (\ref{p}) and (\ref{s}).
Substituting $p$ in (\ref{s}) we obtain
\begin{equation}\label{sa}
s=a\cdot {(1+\theta)s^3-1\over s^3-2}.
\end{equation}
Let $s$ be a solution to (\ref{sa}). Then to find the corresponding $u$ and $v$, by (\ref{sp}) we should solve
$$u+v=s, \ \ uv={1+\theta s^3\over (1+2\theta)s}.$$
This can be reduced to the quadratic equation
\begin{equation}\label{u70}
u^2-su+{1+\theta s^3\over (1+2\theta)s}=0.
\end{equation}
The discriminant of this equation is non-negative iff
$$(1-2\theta)s^3-4\geq 0.$$
From the last inequality we see that $\theta$ and $s$ should satisfy
\begin{equation}\label{ts}
\theta< \theta_c(3):={1\over 2}, \ \ \ s^3\geq {4\over 1-2\theta}.
\end{equation}
Find $a$ from (\ref{sa}):
\begin{equation}\label{aa}
a=s\cdot {s^3-2\over (1+\theta)s^3-1}=:\eta(s).
\end{equation}
We have $\eta'(s)=((1+\theta)s^3-1)^{-2}((1+\theta)s^6+4\theta s^3+2)>0$,  consequently, $a$ is an increasing function of $s$,
for each $s$ satisfying (\ref{ts}),
by (\ref{aa}) we get a unique $a$. The minimal value of $a$ is
$$a_{\min}=\min_{s\geq \sqrt[3]{{4\over 1-2\theta}}}\eta(s)=\eta(\sqrt[3]{{4\over 1-2\theta}})={2\over 3}\sqrt[3]{{4\over 1-2\theta}}.$$
Thus if $a<a_{\min}$ then equation (\ref{aa}) has no solution $s$; if $a\geq a_{\min}$ then equation (\ref{aa}) has a unique solution
$s^*=s(a)$.  For this unique solution from (\ref{u70}) we obtain
\begin{itemize}
\item one value of $u$: $u=u^*$ if $s^*=\sqrt[3]{{4\over 1-2\theta}}$, i.e., $a=a_{\min}$.
In this case by (\ref{u70}), we get $u=v=u^*$, i.e. the system (\ref{rus3.5}) has a unique solution $(u^*,u^*)$.
\item two value of $u$: $u=u_1, u_2$ if  $s^*>\sqrt[3]{{4\over 1-2\theta}}$, i.e., $a>a_{\min}$.
In this case the system (\ref{rus3.5}), outside of the line $u=v$, has exactly two solutions  $(u,v)=(u_1, u_2)$ and $(u_2,u_1)$.
\end{itemize}

Denote
$$\lambda_{\rm cr}(3)=a^3_{\min}={1\over 2-4\theta}\cdot ({4\over 3})^3.$$

Summarize results of this subsubsection in the following
\begin{thm}\label{tk3} For the SCWR-model, in the ferromagnetic case $J>0$, for $k=3$, the following assertions hold:
\begin{itemize}
\item[1)] If $\theta\geq\theta_{c}(3)$ then there exists a unique TISGM.
\item[2)] If $\theta<\theta_{c}(3)$ then
\begin{itemize}
\item[2.a)] if $\lambda\leq\lambda_{{\rm cr}}(3)$ then there is a unique TISGM, denoted by $\mu^*$.
\item[2.b)] if $\lambda>\lambda_{{\rm cr}}(3)$ then there are exactly three TISGMs, denoted by $\mu^*$, $\mu^*_i$, $i=1,2$.
\end{itemize}
\end{itemize}
\end{thm}

{\it SubsubCASE $k\geq 4$.}
Find $v^k$ from the first equation of system (\ref{rus3.5})
and $u^k$ from the second equation:
$$v^k={(a-u)(1+u^k)\over u-a\theta}, \ \ u^k={(a-v)(1+v^k)\over v-a\theta}$$
and using this forms in (\ref{rus3.5}) we get

\begin{equation}\label{rus3.6} \left\{\begin{array}{ll}
u=\gamma(v)\\[2mm]
v=\gamma(u),
\end{array}\right.
\end{equation} where
$$\gamma(u)=a(1+\theta)-u+{u-a\theta\over
1+u^k}.$$
The following lemma is useful.

\begin{lemma}\label{l1} \cite{Kes} Let $f:[a,b]\rightarrow [a,b]$
be a continuous function with a fixed point $\xi \in (a,b)$. We
assume that $f$ is differentiable at $\xi$ and $f^{'}(\xi)<-1.$
Then there exist points $x_0$ and $x_1$, $ a\leq x_0<\xi<x_1
\leq b,$ such that $f(x_0)=x_1$ and $f(x_1)=x_0.$
\end{lemma}
We shall use this lemma for our function $\gamma$.
It is clear that the function $\gamma(x)$ is continuous and
differentiable.

\begin{lemma}\label{l0} For any $x\in [a\theta, a]$ we have $\gamma'(x)<0$.
\end{lemma}
\begin{proof}
We have
\begin{equation}\label{gh}
\gamma'(x)=-{x^{2k}+(k+1)x^k-ka\theta x^{k-1}\over (x^k+1)^2}.
\end{equation}
Note that $\gamma'(x)=0$, i.e. $x^{k+1}+(k+1)x-ka\theta=0$ has a unique positive
root $\hat x$. Indeed, it is well known (see \cite[p.28]{Pra})\footnote{This is known as the Descartes rule: The number of positive roots of the polynomial $p(x)=a_0x^n+a_1x^{n-1}+\dots+a_n$ does not exceed the number of sign changes in the sequence $a_0, a_1,\dots,a_n$.} that the number of positive
roots of a polynomial does not exceed the number of sign
changes of its coefficients. Using this fact one can see that the
equation $\phi(x):=x^{k+1}+(k+1)x-ka\theta=0$ has up to one positive root.
Since $\phi(0)=-ka\theta <0$ and $\phi(a\theta)=(a\theta)^{k+1}+a\theta>a\theta>0$, the equation
has exactly one solution $\hat x\in (0, a\theta)$, i.e. $\hat x<a\theta$.
Therefore $\phi(x)>0$ for all $x\in [a\theta, a]$. This completes the proof.
\end{proof}

By this lemma it follows that the function $\gamma(x)$
is decreasing in $(a\theta, a)$. Moreover, by (\ref{uav}) it is clear
that $\gamma(u)>a\theta$, $\gamma(a\theta)=a$, and
$\gamma(a)=a\cdot{1+\theta a^k\over 1+a^k}<a,$ i.e. $\gamma:[a\theta,a]\rightarrow
[a\theta,a].$ Consequently, the
equation $\gamma(x)=x$ has a unique solution $x=\xi\in (a\theta, a).$

Since $\xi$ is
a fixed point of $\gamma$, we have
\begin{equation}\label{fp}
\xi=a(1+\theta)-\xi+{\xi-a\theta\over 1+\xi^k},
\end{equation}
consequently
\begin{equation}\label{fpp}\xi={a\cdot(1+(1+\theta)\xi^k)\over 1+2\xi^k}.\end{equation}

By (\ref{gh}) we get that $\gamma'(\xi)<-1$ can be written as
$$(k-1)\xi^k-ka\theta \xi^{k-1}-1>0,$$
in this inequality using (\ref{fpp}) we get the following polynomial inequality:
 $$[k-1-(k+1)\theta]\xi^{2k}+[k-2-(k+1)\theta]\xi^k-1>0.$$
To simplify formulas, we introduce
$$t=\xi^k, \ \ A=k-1-(k+1)\theta.$$
Then the last inequality can be written as
\begin{equation}\label{ta}
At^2+(A-1)t-1>0.
\end{equation}
The discriminant of the quadratic inequality is positive:
$$D=(A-1)^2+4A=(A+1)^2>0.$$
Therefore (\ref{ta}) can be written as (since $t>0$)
$$(t+1)(At-1)>0 \ \ \ \Rightarrow \ \ At>1.$$
The last inequality has solution iff
\begin{equation}\label{the}
A>0 \ \ \Rightarrow \ \ \theta<\theta_{\rm c}(k):={k-1\over k+1}.
\end{equation}
Then $t>{1\over A}$.
Hence we get
\begin{equation}\label{sk}
\xi>{1\over \sqrt[k]{A}}={1\over \sqrt[k]{k-1-(k+1)\theta}}.\end{equation}

From the equation $\gamma(\xi)=\xi$, we have
$$a={\xi+2\xi^{k+1}\over 1+(1+\theta)\xi^k}=:\varphi(\xi).$$
Note that
$$\varphi'(\xi)=1+{(1-\theta)\xi^k\over 1+(1+\theta)\xi^k}\left[1+{k\over 1+(1+\theta)\xi^k}\right]>0,$$
i.e. the function $\varphi(\xi)$ is increasing. So
$$a_{\min}=\min_{{\xi\geq {1\over \sqrt[k]{A}}}}\varphi(\xi)=\varphi({1\over \sqrt[k]{A}}),$$
this implies that
\begin{equation}\label{lam}
\lambda_{\min}=a_{\min}^k=\lambda_{\rm cr}(k):={1\over k-1-\theta(k+1)}\cdot\left({k+1\over k}\right)^k.
\end{equation}
Hence by Lemma \ref{l1} if $\lambda>\lambda_{\rm cr}$ then the system
(\ref{rus3.5}) has at least three solutions $(\xi,\xi), \ (x_0,y_0)$ and
$(y_0,x_0)$.

\begin{lemma} If $\theta\geq {k-1\over k}$ or $\theta<{k-1\over k}$ and $\lambda< {1\over k-1-k\theta}$, then
$$\lim_{n\rightarrow\infty}\gamma^{(n)}(\xi_0)=\xi$$
for any $\xi_0\in[a\theta, a]$, where $\gamma^{(n)}$ is $n$-iteration of
map $\gamma.$
\end{lemma}
\begin{proof} By Lemma \ref{l0} we have that $\gamma'(x)<0$. Now $-1\leq \gamma'(x)$ is equivalent to
 $$\psi(x):=(k-1)x^k-ka\theta x^{k-1}-1\leq 0.$$
We should solve this inequality for $x\in [a\theta, a].$ Note that 
$\psi'(x)=k(k-1)x^{k-2}(x-a\theta)>0$, in $(a\theta,a]$, i.e. $\psi$ is 
an increasing function. It is clear that $\psi(x)=0$ has a unique positive solution, denote it by $\tilde x$. Moreover, 
$\psi(x)<0$, if $x<\tilde x$.
We have $\psi(a\theta)=-(a\theta)^k-1<0$. Note that if $\tilde x > a$, then
$\gamma'(x)\in (-1,0)$ for each $x\in [a\theta, a]$, i.e. $\gamma$ is contracting function.
From $\tilde x > a$, then we take $\psi(a)<0$, i.e.
$$(k-1)a^k-ka\theta a^{k-1}-1=(k-1-k\theta)a^k-1< 0.$$
It is easy to see that the last inequality is true iff the conditions of the lemma are satisfied (recall $\lambda=a^k$).
\end{proof}
From this Lemma it follows that under its conditions the equation $\gamma(\gamma(x))=x$ has the unique solution
$x=\xi$.
Denote
\begin{equation}\label{thla}
\theta_{\rm c}'(k):={k-1\over k}, \ \ \ \lambda_{\rm cr}'(k):= {1\over k-1-k\theta}.
\end{equation}

Thus, we proved the following

\begin{thm}\label{tkk} Consider the ferromagnetic model, and let $k\geq 4$. Then
\begin{itemize}
\item[1)] If $\theta\geq \theta'_{\rm c}(k)$ or $\theta< \theta'_{\rm c}(k)$ and $\lambda<\lambda'_{\rm cr}(k)$ then there exists a unique TISGM.

\item[2)] If $\theta'_c(k)>\theta\geq \theta_{\rm c}(k)$ or $\theta< \theta_{\rm c}(k)$ and $\lambda_{\rm cr}'(k)<\lambda\leq\lambda_{\rm cr}(k)$ then there exists at least one TISGM.

\item[3)] For  $\theta< \theta_{\rm c}(k)$ and $\lambda>\lambda_{\rm cr}(k)$ there exist at least three
TISGMs.
The critical values are defined in (\ref{the}), (\ref{lam}) and (\ref{thla}).
\end{itemize}
\end{thm}
\begin{figure}[!htb]
	\includegraphics[width=0.49\textwidth]{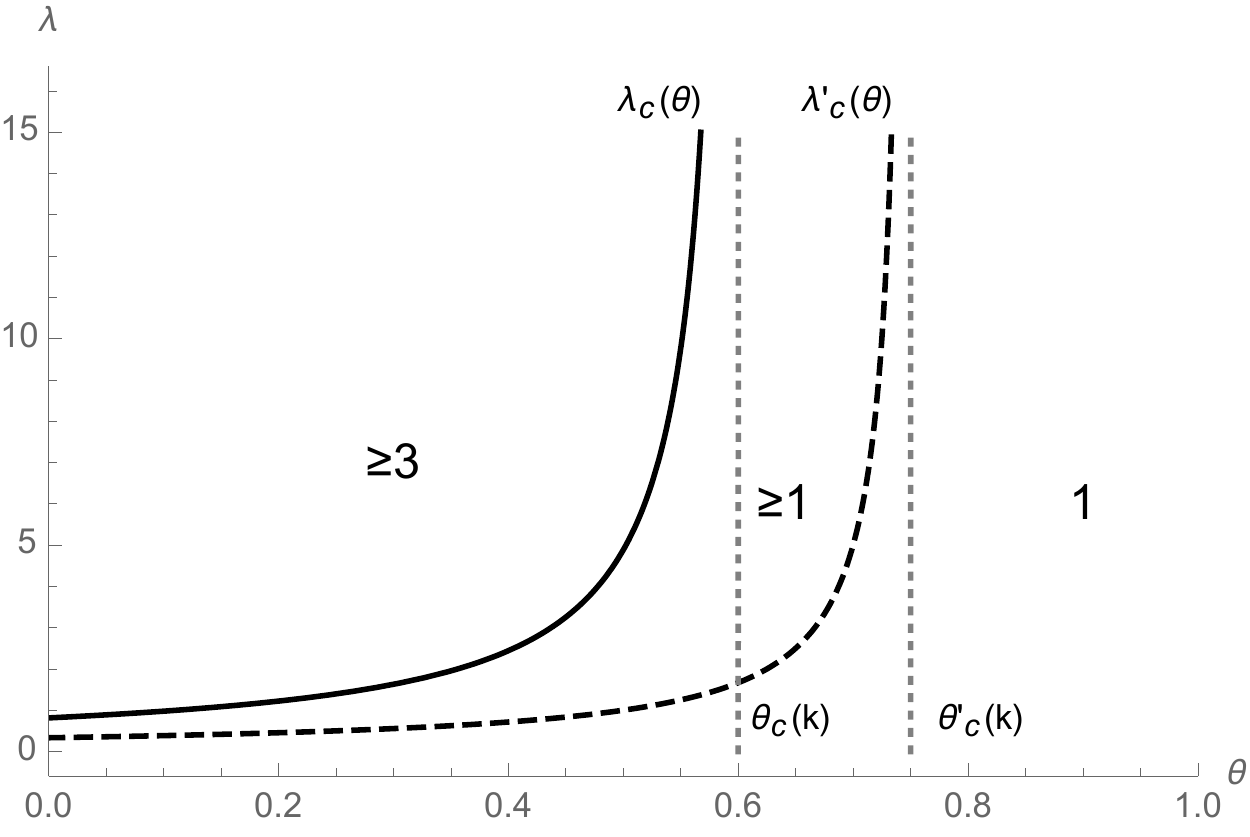}
	\includegraphics[width=0.49\textwidth]{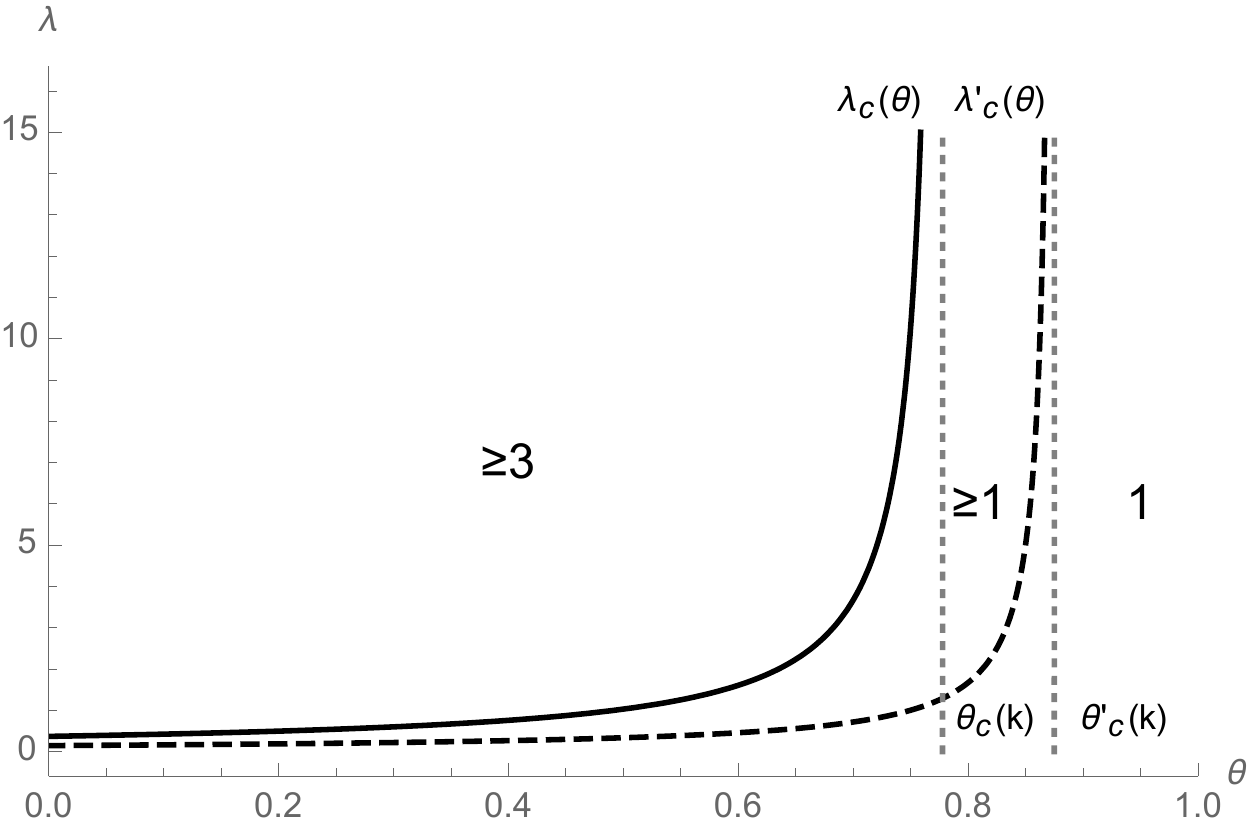}
	\caption{Curves of critical $\lambda$'s for $k=4$ (left) and $k=8$ (right). The numbers shown on the domains correespond to the number of TISGMs.}
\end{figure}

We have the following conjecture
\begin{con}\label{con} In the part 2) (resp. 3)) of Theorem \ref{tkk} the numbers of TISGM
 is  exactly one (resp. three).
\end{con}
{\it An argument towards to a proof of Conjecture \ref{con}:}
 Note that the critical values mentioned in Theorem \ref{tkk} coincide with values given in
 Theorem \ref{t<} and Theorem \ref{tk3} for $k=2, 3$ respectively, in these theorems the number of TISGMs
is exactly one or three, i.e. Conjecture is true for $k=2,3$.  The following argument shows that the Conjecture should be true.
 System (\ref{rus3.5}) can be written as linear system of equations with respect to $a$ and $T=a\theta$:
\begin{equation}\label{r5} \left\{\begin{array}{ll}
(1+ u^k)a+v^k T=u(1+u^k+v^k),\\[3 mm]
(1+ v^k)a+u^k T=v(1+u^k+v^k).
\end{array}\right.
\end{equation}
The solution of it after dividing to $u-v$ is
$$a=a(u,v)={\sum_{i=0}^ku^{k-j}v^j\over \sum_{i=0}^{k-1}u^{k-1-j}v^j}, $$
$$T=T(u,v)={uv\left(\sum_{i=0}^{k-2}u^{k-2-j}v^j\right)-1\over \sum_{i=0}^{k-1}u^{k-1-j}v^j}. $$
Now take $a$ fixed and consider the Lagrange multiplier method to find minimal value of $T(u,v)$. Then we should solve
\begin{equation}\label{Lm}\left\{\begin{array}{lll}
{\partial T(u,v)\over \partial u}=\ell {\partial a(u,v)\over \partial u},\\[2mm]
{\partial T(u,v)\over \partial v}=\ell {\partial a(u,v)\over \partial v},\\[2mm]
a=a(u,v)\end{array}.\right.
\end{equation}
This is complicated to solve, but by symmetry of $u,v$ the first and second equation should have a solution $u=v$.
Using $u=v$ in the third equation, i.e. $a=a(u,u)$ we get $u={ak\over k+1}$.
Then taking $u=v={ak\over k+1}$ from $T=T(u,v)$ we find
$$a=\sqrt[k]{\lambda}={k+1\over k}\cdot {1\over \sqrt[k]{k-1-\theta(k+1)}}.$$
Thus we get exactly critical values of $\theta$ and $\lambda$.
But the only \textbf{remaining problem is} to show that function $a$ has its constrained
minimum on the solution of the system (\ref{Lm}) with $u=v$.
Numerical analysis by ``Mathematica" showed that the last statement is true for small values of $k=4, \dots, 15.$ \\

\begin{rk} \textcolor{white}{white}
	\begin{itemize}
\item[1.] We have
$$\theta'_{\rm c}(k)<\theta_{\rm c}(k), \ \ \ \lambda_{\rm cr}'(k)<\lambda_{\rm cr}(k),$$
$$\lim_{k\to \infty}\theta_{\rm c}(k)-\theta'_{\rm c}(k)=0, \ \ \ \lim_{k\to\infty}\lambda_{\rm cr}(k)-\lambda'_{\rm cr}(k)=0.$$

\item[2.] For the mean-field version of the Soft-Core model the behavior is similar but there is no critical value of $\lambda$. It is proven that no phase transition occurs for $\theta  \leq \theta_c(\lambda) := \exp(-2-\frac{e}{\lambda})$ and if $\theta>\theta_c(\lambda)$ then there exist multiple Gibbs measures (see \cite{KK1}).
\end{itemize}
\end{rk}

\section{Lower and upper bounds of solutions to the functional equation (\ref{eq: comp})}

Introduce  a new function
$$F(x,y,\theta):={1+x+\theta y\over 1+x+y}.$$

Rewrite the system of equations  (\ref{eq: comp}) in the
following form:
\begin{equation}\label{rus2.3}\begin{array}{llllll}
z_{1,i}=\lambda \prod_{j\in S(i)}F(z_{1,j},z_{2,j},\theta),\\[4mm]
z_{2,i}=\lambda \prod_{j\in S(i)}F(z_{2,j},z_{1,j},\theta),
\end{array}
\end{equation}
 where
$$z_{1,i}=\exp(\tilde h_{+,i}), \ \  z_{2,i}=\exp(\tilde h_{-,i}).$$

\begin{pro}\label{rusp1} Let $\theta<1$. If $(z_{1,i}, z_{2,i})$ is a solution
of (\ref{rus2.3}) then
$$z^-_j\leq z_{j,i} \leq z^+_j,$$
for any $j=1, 2, \ \ x\in \mathbb T^k,$ where $(z^-_1, z^+_1, z^-_2,
z^+_2)$ is a solution of
\begin{equation}\label{rus3.7}
\left\{\begin{array}{llll}
z^-_1=\lambda (F(z^-_{1},z^+_{2},\theta))^k,\\[2mm]
z^+_1=\lambda(F(z^+_{1},z^-_{2},\theta))^k,\\[2mm]
z^-_2=\lambda(F(z^-_{2},z^+_{1},\theta))^k,\\[2mm]
z^+_2=\lambda(F(z^+_{2},z^-_{1},\theta))^k.
\end{array}\right.
\end{equation}

\end{pro}
\begin{proof} First we note that $z_{j,i}>0$, $j=0,1$ and for the function $F$ if $x>0$ and $y>0$ then
$$\theta< F(x,y,\theta)<1.$$
Consequently, from (\ref{rus2.3}) we get
\begin{equation}\label{n11}
\begin{array}{ll}
z^-_{1,1}:=\lambda \theta^k<z_{1,i}<\lambda=: z^+_{1,1}.\\[2mm]
z^-_{2,1}:=\lambda \theta^k<z_{2,i}<\lambda=: z^+_{2,1}.
\end{array}
\end{equation}

It is easy to see that the function $F$, for $\theta<1$, is monotone increasing
with respect to $x$, but monotone decreasing with respect to $y$. Using this property, and
the bounds (\ref{n11}) we get from (\ref{rus2.3}) that
\begin{equation}\label{n22}
\begin{array}{ll}
z^-_{1,2}:=\lambda (F(z^-_{1,1},z^+_{2,1},\theta))^k<z_{1,i}<\lambda (F(z^+_{1,1},z^-_{2,1},\theta))^k=: z^+_{1,2}.\\[2mm]
z^-_{2,2}:=\lambda (F(z^-_{2,1},z^+_{1,1},\theta))^k<z_{2,i}<\lambda (F(z^+_{2,1},z^-_{1,1},\theta))^k=: z^+_{2,2}.
\end{array}
\end{equation}
Now iterating this argument we get
\begin{equation}\label{nnn}
\begin{array}{ll}
z^-_{1,n+1}:=\lambda (F(z^-_{1,n},z^+_{2,n},\theta))^k<z_{1,i}<\lambda (F(z^+_{1,n},z^-_{2,n},\theta))^k=: z^+_{1,n+1}.\\[2mm]
z^-_{2,n+1}:=\lambda (F(z^-_{2,n},z^+_{1,n},\theta))^k<z_{2,i}<\lambda (F(z^+_{2,n},z^-_{1,n},\theta))^k=: z^+_{2,n+1}.
\end{array}
\end{equation}
One can see that $z^-_{i,n}$,  (resp. $z^+_{i,n}$), $i = 1, 2$ are increasing (decreasing) and bounded sequences. Thus there exist
$$\lim_{n\to\infty}z^{\pm}_{i,n}=z^{\pm}_i, \ \ i=1,2.$$
\end{proof}

As in the case wrench (see \cite{Br1}) we can prove the following statements:

\begin{pro}\label{rusp2} If $z=(z_1^-, z^+_1, z^-_2, z^+_2)$ a
solution of (\ref{rus3.7}) then $z^-_1=z^+_1$ iff $z^-_2=z^+_2$.
\end{pro}
\begin{proof} Assume $z^-_1=z^+_1$ then from the first and second equations of
(\ref{rus3.7}) we get $F(z^-_{1},z^+_{2},\theta)=F(z^-_{1},z^-_{2},\theta)$, consequently $z^-_2=z^+_2$.
If now $z^-_2=z^+_2$ then from the third and fourth equations we get $F(z^-_{2},z^-_{1},\theta)=F(z^-_{2},z^+_{1},\theta)$,
i.e. $z^-_1=z^+_1$.
\end{proof}
This proposition is very useful:
\begin{cor}\label{rusc1}
 If the system (\ref{rus3.7}) has a unique solution then
system (\ref{rus2.3}) also has a unique solution. Moreover, this solution
is the unique solution of (\ref{exy}).
\end{cor}

Now we shall find exact values of $z^-_i, z^+_i, i=1,2$ for $k=2.$

Consider the system consisting of the first and last equations
in system (\ref{rus3.7}):
\begin{equation}\label{ru}
\left\{\begin{array}{llll}
z^-_1=\lambda (F(z^-_{1},z^+_{2},\theta))^k,\\[2mm]
z^+_2=\lambda(F(z^+_{2},z^-_{1},\theta))^k.
\end{array}\right.
\end{equation}
Taking $x=z^-_1$ and $y=z^+_2$ one can see that this system coincides with the system (\ref{exy}).
Therefore for $k=2$ we have $(z^-_1, z^+_2)\in \mathcal M$ (see (\ref{ye})).

Similarly,
from the second and third equalities of (\ref{rus3.7}) we get
$(z^+_1, \, z^-_2) \in \mathcal M$.
Thus, we have the following

\begin{pro}\label{rusp3} If $k=2$, $\theta<1$ then
\begin{itemize}

\item[1)] for $\theta\geq 1/3$ or $\theta<1/3$ and $\lambda\leq {9\over 4(1-3\theta)}$ the system (\ref{rus3.7}) has unique
solution $(x^*,x^*,x^*,x^*)$;

\item[2)]  for $\theta<1/3$ and $\lambda> {9\over 4(1-3\theta)}$, the system (\ref{rus3.7}) has four solutions, (as vector $(z^-_1, z^+_1, z^-_2, z^+_2)$):
$$(x_1^*,x_1^*,x_2^*,x_2^*),\ \ (x_1^*,x_2^*,x_1^*,x_2^*), \ \ (x_2^*,x_1^*,x_2^*,x_1^*), \ \ (x_2^*,x_2^*,x_1^*,x_1^*) $$ where coordinates are given
in $\mathcal M$.
\end{itemize}
\end{pro}

\begin{cor}\label{rusc2} If $k=2$, $\theta<1/3$ and  $\lambda>{9\over 4(1-3\theta)}$ then for any
solution $(z_{1,i}, z_{2,i})$ of (\ref{rus2.3}) we have
$$x^*_1\leq z_{j,i}\leq x_2^*, \ \ j=1,2, \ \ \forall i\in \mathbb T^2.$$
\end{cor}
\begin{rk} Since we also have explicit
 formulas for solutions of (\ref{rus2.3}) in case $k=3$, one can similarly obtain
 an analogue of Corollary \ref{rusc2} in the case $k=3$.

\end{rk}

It is known that for each $\beta>0$ the set of Gibbs measures form a
non-empty convex compact set in the space of all
probability measures on $\Omega$ endowed with the weak topology
(see, e.g., \cite[Chapter~7]{Ge}). A Gibbs measure is
called extreme if it cannot be expressed as convex combination of other measures.
The crucial observation is that according to \cite[Theorem 12.6]{Ge}, \emph{any
extreme Gibbs measure is a splitting GM}; therefore, the
question of uniqueness of Gibbs measures is reduced to that in
the class of splitting GMs.

For two configurations $\sigma^1, \sigma^2\in\Omega$,
of the WR-model,
we write $\sigma^1\le\sigma^2$ if $\sigma^1(x)\leq\sigma^2(x)$
for all $x\in V$.
This partial order defines the concept of a
monotone increasing (decreasing):
A function $f:\Omega\to \mathbb R$ is said to be
\emph{monotone increasing} if $f(\sigma^{1})\le f(\sigma^{2})$
whenever $\sigma^1\leq\sigma^2$. For two probability measures
$\mu_1$, $\mu_2$ on $\Omega$, we write $\mu_1\leq\mu_2$ if $\int
f\,d{\mu}_1\leq\int f\,d{\mu}_2$ for any monotone increasing
$f$. It is known that the Gibbs measures $\mu^*_1$, $\mu^*_2$, corresponding
to ``extreme'' boundary laws, (in our case $x^*_1$ and $x^*_2$ according
to Corollary \ref{rusc2})  enjoy the following monotonicity property:
$\mu^{*}_1\leq\mu\leq
\mu^{*}_2$ for any Gibbs measure
$\mu$ (not necessarily splitting) see Bibliographical notes of
\cite[Chapter 2]{Ge} for more details.
Thus we have the following

 \begin{thm}
 The splitting Gibbs measures
$\mu^{*}_1$  and  $\mu^{*}_2$, (mentioned in Theorem \ref{t<} and Theorem \ref{tk3}) are extreme.
\end{thm}

\section{Periodic GMs in ferromagnetic model}
In general periodicity of a splitting Gibbs measure can be defined by the group representation of the
Cayley tree (see \cite{Ro}). In our model there are only periodic measures with period
two (can be shown as Theorems 2.3 and 2.4 in \cite{Ro} and Theorem 2 of \cite{MRS}).
Namely, to construct two-periodic points we separate vertices of the Cayley tree to odd and even ones:
 A vertex is called odd (resp. even) if it is at odd (resp. even) distance from the root $0$.
 Then a two-periodic splitting GM corresponds to a solution $z_i=(z_{1,i}, z_{2,i})$ of (\ref{rus2.3}) having the form
 $$ z_i=\left\{\begin{array}{ll}
 z=(z_1,z_2), \ \ \mbox{if} \ \ i \ \ \mbox{is even}\\[2mm]
 t=(t_1,t_2), \ \ \mbox{if} \ \ i \ \ \mbox{is odd}.
 \end{array}\right.
 $$
 Dividing the tree into even and odd sites is analogous 
 to a checkerboard decomposition of the lattice sites on $\Z^d$. 
   Then from (\ref{rus2.3}) we get the following system of equations:
\begin{equation}\label{rus4.5}
z_1=\lambda\bigg({1+t_1+\theta t_2\over 1+t_1+t_2}\bigg)^k,\ \
z_2=\lambda\bigg({1+t_2+\theta t_1\over 1+t_1+t_2}\bigg)^k,
\end{equation}
$$
t_1=\lambda\bigg({1+z_1+\theta z_2 \over 1+z_1+z_2}\bigg)^k,\ \
t_2=\lambda\bigg({1+z_2+\theta z_1\over 1+z_1+z_2}\bigg)^k.
$$
We solve this system in a simple case: assuming $z_1=z_2=z$, $t_1=t_2=t$ then (\ref{rus4.5}) reduces to the following system
\begin{equation}\label{rus4.6}
z=\lambda\bigg({1+(1+\theta)t\over 1+2t}\bigg)^k,\ \
t=\lambda\bigg({1+(1+\theta)z\over 1+2z}\bigg)^k.
\end{equation}
Denote
$$\phi(x)=\lambda\left({1+(1+\theta)x\over 1+2x}\right)^k.$$
 Then from
(\ref{rus4.6}) we have
\begin{equation}\label{rus4.7}
z=\phi(t),\ \ t=\phi(z).
\end{equation}

As it was shown above the equation $x=\phi(x)$ has unique solution
$x^*=x^*(k,\theta, \l),$ for any $k\geq 1$, $\theta<1$ and $\l>0.$

Now for the equation $\phi(\phi(x)) = x$ we use Lemma \ref{l1}.

Using bounds given in Proposition \ref{rusp1} we can apply the lemma in $[z_1^-, z_2^+]$.
Denote
$$s^{\pm}:=s^{\pm}(k,\theta):={k-3-(k+1)\theta\pm\sqrt{(1-\theta)[k^2-6k+1-(k+1)^2\theta]}\over 4(1+\theta)}.$$
$$\lambda^{\pm}(k,\theta)=s^{\pm}\cdot \left({1+2s^{\pm}\over 1+(1+\theta)s^{\pm}}\right)^k.$$
\begin{thm}\label{rust6} For $k\geq 6$, $\theta<{k^2-6k+1\over (k+1)^2}$ and $\lambda\in (\lambda^{-}(k,\theta), \lambda^{+}(k,\theta))$
there are at least three 2-periodic splitting Gibbs measures $\mu_0,$ $\mu_*,$ $
\mu_1$. These correspond to three solutions $(x_0,x_1),$ $
(x_*,x_*),$ $(x_1,x_0)$ of (\ref{rus4.7}).
\end{thm}

\begin{rk} The densities to see holes on the even (respectively odd) sites are strictly 
different, as a simple computation relating boundary laws to single-site marginals of the 
measures shows. 
This underlines that the ferromagnetic 
Soft-Core model is substantially richer than the ferromagnetic Ising model, 
as in the latter such types of states do not exist (while they do exist in the antiferromagnetic 
Ising model.) 
\end{rk}

\begin{proof}  Note that function $\phi(x)$ is decreasing for any
$x>0$. By Lemma \ref{l1}, if $x^*$ satisfies
\begin{equation}\label{rus4.9}
\phi(x^*)=x^*,\ \ \phi(x^*)<-1,
\end{equation}
then (\ref{rus4.6}) has two solutions. From (\ref{rus4.9}), 
using that
$$\left({1+(1+\theta)x^*\over 1+2x^*}\right)^{k-1}={x^*\over \lambda} \cdot {1+2x^*\over 1+(1+\theta)x^*},$$
we get  
\begin{equation}\label{rus4.10}
2(1+\theta)(x^*)^2+(3+\theta-k(1-\theta))x^*+1<0.
\end{equation}
By an analysis of this inequality one can see that it has solution 
$x^*\in (s^-, s^+)$ iff the 	conditions of theorem are satisfied. 
From $\phi(x^*)=x^*$ we get 
\begin{equation}\label{kap}
\lambda=\kappa(x^*):=x^* \cdot \left({1+2x^*\over 1+(1+\theta)x^*}\right)^k.
\end{equation}

We have (since $\theta<1$)
$$\kappa'(x)={(1+2x)^{k-1}\over (1+(1+\theta)x)^{k+1}}\cdot [(1+2x)(1+(1+\theta)x)+(1-\theta)kx]>0.$$
Therefore, from (\ref{kap}), by $x^*\in (s^-, s^+)$ we get 
$$\lambda\in (\kappa(s^-), \kappa(s^+))=(\lambda^{-}(k,\theta), \lambda^{+}(k,\theta)).$$
This completes the proof.
\end{proof}
\begin{figure}[!htb]
	\includegraphics[width=0.7\textwidth]{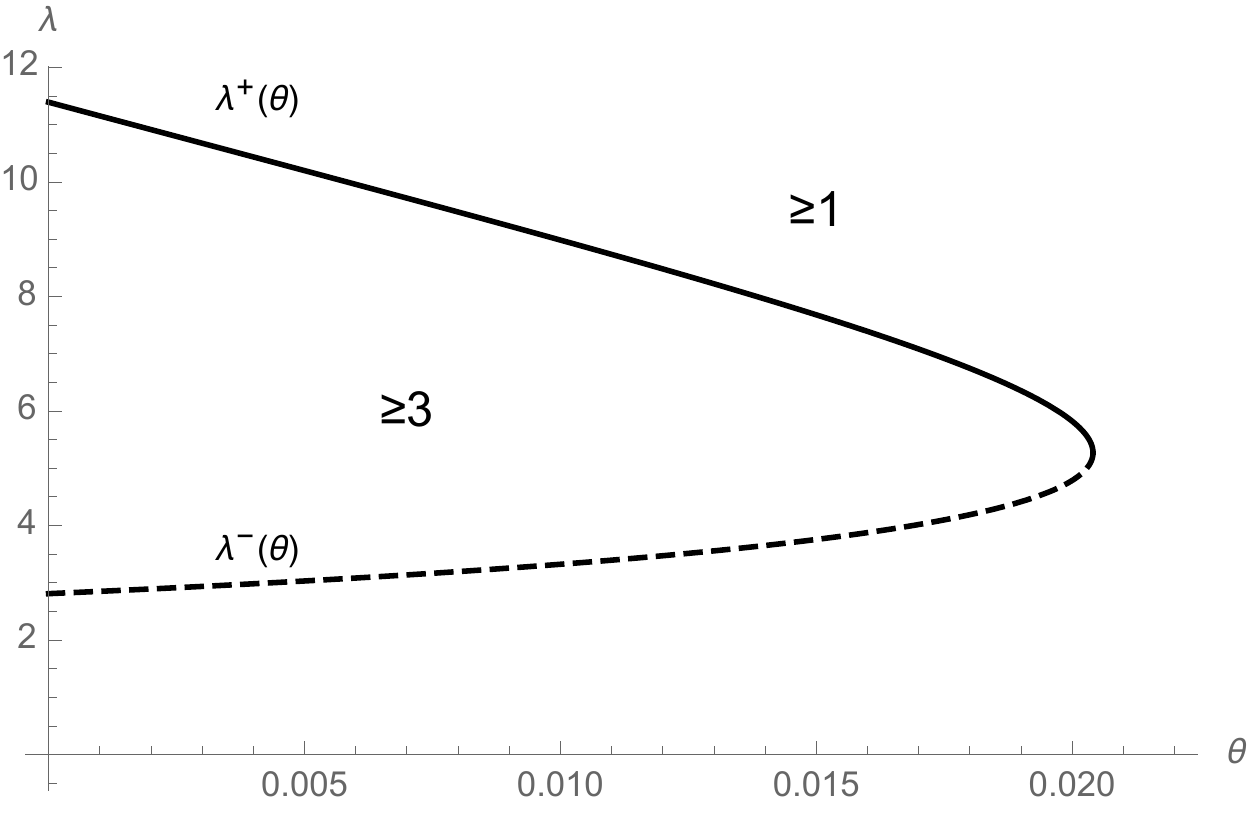}
	\caption{Curves of $\lambda^+$ and $\lambda^-$ for $k=6$. Inequalities on the picture corrpesond to the number of periodic SGMs on the domains.}
\end{figure}
\section{Non-Periodic splitting GMs}

In this subsection we shall use a construction similar to the Bleher-Ganikhodjaev construction \cite{Bl1}.

Recall that (\ref{eq: comp}) has the following form
\begin{equation}\label{rus4.1}
\begin{array}{ll}
h_{1,x}=\ln\lambda +\sum_{y\in S(x)}f(h_{1,y}, h_{2,y}, \theta),\\[3mm]
h_{2,x}=\ln\lambda +\sum_{y\in S(x)}f(h_{2,y}, h_{1,y}, \theta).
\end{array}
\end{equation}
where $f(x,y,\theta)=\ln\left({1+e^x+\theta e^y\over 1+e^x+e^y}\right)$.

The following lemma is simple (see Lemma 9 in \cite{Ro13}):

\begin{lemma}\label{lb} The following estimates hold for every $(x,y)\in \mathbb R^2$:
$$\left|{\partial f(x,y,\theta)\over \partial x}\right|\leq {|1-\sqrt{\theta}|\over 1+\sqrt{\theta}},
 \ \ \left|{\partial f(x,y,\theta)\over \partial y}\right|\leq {|1-\sqrt{\theta}|\over 1+\sqrt{\theta}},$$

$$|f(x,y,\theta)-f(u,y,\theta)|\leq {|1-\sqrt{\theta}|\over 1+\sqrt{\theta}}|x-u|.$$
\end{lemma}
We show that the system of equations (\ref{rus4.1}) has uncountably many
non-translational-invariant solutions.

Take an arbitrary infinite
path $\pi=\{x^0=x_0<x_1<x_2<...\}$ on the Cayley tree starting at
the origin $x_0=x^0$. Establish a 1-1 correspondence between such
paths and real numbers $t\in [0,1]$. Write $\pi=\pi(t)$
when it is desirable to stress the dependence upon $t$. Map path
$\pi$ to a vector function $h^{\pi}:\;x\in V\mapsto h_x^{\pi}$ satisfying
(\ref{rus4.1}). Note that $\pi$ splits the Cayley tree $\Gamma^k$ into two subgraphs
$\Gamma_1^k$ and $\Gamma_2^k$.

Under the non-uniqueness conditions of Theorem \ref{tkk} the function $h^{\pi}$ is  defined by
\begin{equation}\label{rus5.1}
h_{x}^{\pi}=\left\{\begin{array}{ll}
(\ln(x_1^*), \ln(x_2^*)), \ \ \mbox{if} \ \ x\in \Gamma^k_1,\\[3mm]
(\ln(x_2^*), \ln(x_1^*)), \ \ \mbox{if}\ \  x\in \Gamma^k_2,
\end{array}\right.
\end{equation}
where $x^*_1$ and $x^*_2$ are distinct solutions of the system (\ref{rus3.5}).

If ${2(1-\sqrt{\theta})\over 1+\sqrt{\theta}}<1$, i.e., $\theta>{1\over 9}$ then
one can use Lemma \ref{lb} to prove the following (see \cite{Bl1} or Section 2.6 in \cite{Ro}).

\begin{thm}\label{rust7}  If $k\geq 2$, ${1\over 9}<\theta<\theta_c={k-1\over k+1}$ and $\l>\lambda_{\rm cr}(k)$ then for any infinite path $\pi$ there
exists a unique function $h^{\pi}$ satisfying (\ref{rus4.1})  and
(\ref{rus5.1}).
\end{thm}

The vector functions $h^{\pi(t)}$ are different for different  $t \in [0;1].$
Now let $\mu(t)$ denote the splitting Gibbs measure corresponding to function
$h^{\pi(t)}$, $t\in [0;1].$ Thus we have the following:

\begin{thm}\label{rust8} If conditions of Theorem \ref{rust7} are satisfied then
for any  $t\in [0;1]$, there exists an extreme Gibbs
measure $\mu(t)$. Moreover, the splitting Gibbs measures $\mu^*_1$, $\mu^*_2$ (see
Theorem \ref{t<} and \ref{tk3}) are specified as $\mu(0)=\mu^*_1$ and
$\mu(1)=\mu^*_2.$
\end{thm}

Note that the measures $\mu(t)$ are different for different $t\in [0,1]$, therefore
we obtain a continuum of distinct splitting Gibbs measures which are
non-periodic.

 \section*{ Acknowledgements}

S. Kissel and U.A. Rozikov thank the  RTG 2131, the research training
group on High-dimensional Phenomena in Probability - Fluctuations and Discontinuity
 and the Ruhr-University Bochum (Germany).

\end{document}